\newif\ifHYPER\global\HYPERtrue
\newif\ifINTERNAL\global\INTERNALfalse
\documentclass[10pt,b5paper,fleqn]{article}
\usepackage[utf8]{inputenc}
\usepackage{amsmath,amsfonts,amssymb}
\usepackage[english,greek]{babel}
\usepackage[mathscr]{eucal}
\usepackage{color}
\usepackage{graphicx}
\usepackage{amsthm}
\usepackage[all,cmtip]{xy}
\usepackage{dcolumn}

\ifHYPER
\definecolor{ks-green}{rgb}{0.0,0.7,0.0}
\definecolor{ks-red}{rgb}{0.7,0.0,0.0}
\definecolor{ks-blue}{rgb}{0.0,0.0,0.7}
\usepackage{hyperref}
\hypersetup{colorlinks=true,citecolor=ks-green,linkcolor=ks-red}
\fi

\numberwithin{equation}{section}
\theoremstyle{plain}
\newtheorem{theorem}[equation]{Theorem}
\newtheorem{lemma}[equation]{Lemma}
\newtheorem{corollary}[equation]{Corollary}
\newtheorem{proposition}[equation]{Proposition}

\theoremstyle{definition}
\newtheorem{definition}[equation]{Definition}
\newtheorem{Example}[equation]{Example}

\newenvironment{example}{\emph{Example.}}{}
\newenvironment{remark}{\emph{Remark.}}{}

\newenvironment{notation}{\emph{Notation.}}{}

\newcolumntype{d}[1]{D{.}{.}{#1}}

\newcommand{\tsfrac}[2]{\textstyle{\frac{#1}{#2}}}
\newcommand{\mtxstrut}{\rule[-0.5ex]{0pt}{2.8ex}}

\newcommand{\tabstrut}{\rule[-0.5ex]{0pt}{3.2ex}}

\def\trp{^{\!\top}}
\def\inv{^{-1}}
\def\eps{\varepsilon}
\def\dd{\mathrm{d}}     
\def\ncpd{\partial}     

\newcommand\numN{\mathbb{N}}   
\newcommand\numQ{\mathbb{Q}}   
\newcommand\numR{\mathbb{R}}   
\newcommand\numC{\mathbb{C}}   

\newcommand{\fronorm}[1]{\lVert #1 \rVert_\text{F}}
\newcommand{\alphabet}[1]{\mathcal{#1}}
\newcommand{\freeALG}[2]{#1\langle #2\rangle}
\newcommand{\freeFLD}[2]{#1(\!\langle #2\rangle\!)}
\newcommand{\field}[1]{\mathbb{#1}}
\newcommand{\als}[1]{\mathcal{#1}}

\newcommand{\aclo}[1]{\overline{#1}}
\newcommand{\Grad}{\nabla}

\DeclareMathOperator{\rank}{rank}

\DeclareMathOperator{\linsp}{span}

\begin{document}
\selectlanguage{english}
\title{Free (rational) Derivation}
\author{Konrad Schrempf%
  \footnote{Contact: math@versibilitas.at (Konrad Schrempf),
    \url{https://orcid.org/0000-0001-8509-009X},
    Austrian Academy of Sciences, Acoustics Research Institute,
    Wohllebengasse 12--14, 1040 Vienna, Austria.
    }
    \hspace{0.2em}\href{https://orcid.org/0000-0001-8509-009X}{%
    \includegraphics[height=10pt]{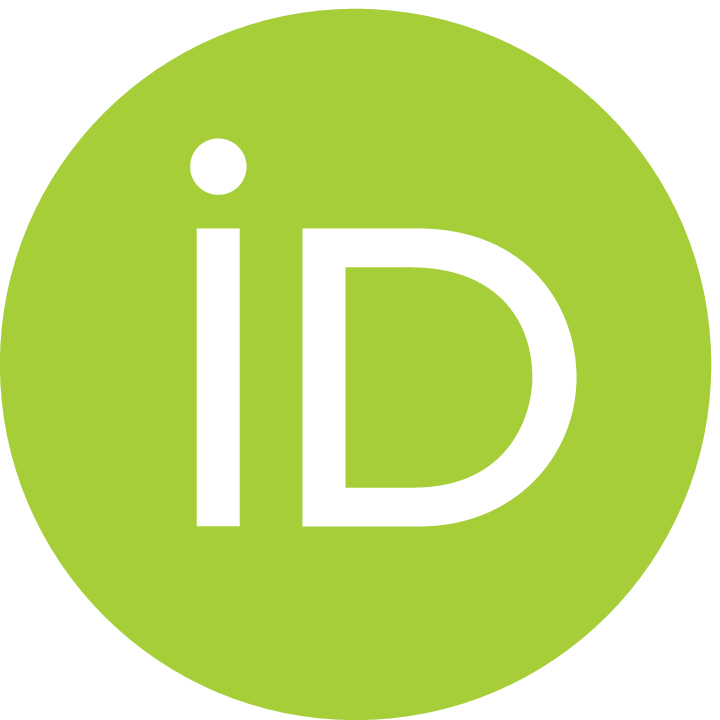}}
  }

\maketitle

\begin{abstract}
By representing elements in \emph{free fields} (over a \emph{commutative}
field and a \emph{finite} alphabet) using Cohn and Reutenauer's
\emph{linear representations}, we provide an algorithmic
construction for the (partial) non-commutative (or Hausdorff-)
\emph{derivative}
and show how it can be applied to the non-commutative version
of the Newton iteration to find roots of matrix-valued
rational equations.
\end{abstract}

\medskip
\emph{Keywords and 2020 Mathematics Subject Classification.}
Hausdorff derivative, free associative algebra, free field, 
minimal linear representation, admissible linear system,
free fractions, chain rule, Newton iteration;
Primary 16K40, 16S85;
Secondary 68W30, 46G05

\section*{Introduction}

Working symbolically with matrices requires \emph{non-commuting}
variables and thus non-com\-mu\-tative (nc) rational expressions.
Although the (algebraic) construction of \emph{free fields},
that is, universal fields of fractions of \emph{free associative algebras},
is available due to Paul M.~Cohn since 1970
\cite[Chapter~7]{Cohn2006a}
, its practical application in terms of \emph{free fractions}
\cite{Schrempf2018c2}
\ ---building directly on Cohn and Reutenauer's \emph{linear representations}
\cite{Cohn1999a}
--- in computer algebra systems is only at the very beginning.

The main difficulty for arithmetic
---or rather \emph{lexetic} from the non-existing Greek word
\selectlanguage{greek}%
lexhtikos \selectlanguage{english}%
(from \selectlanguage{greek}%
lexis \selectlanguage{english}%
for \emph{word}) as analogon to \selectlanguage{greek}%
arijmhtikos \selectlanguage{english}%
(from \selectlanguage{greek}%
arijmos \selectlanguage{english}%
for \emph{number})---
was the construction of \emph{minimal} linear representations
\cite{Schrempf2018a9}
, that is, the \emph{normal form} of Cohn and Reutenauer
\cite{Cohn1994a}
.

Here we will show that free fractions also provide a
framework for ``free'' derivation, in particular of nc polynomials.
The construction we provide generalizes the univariate (commutative)
case we are so much used to, for example
\begin{displaymath}
f = f(x) = x^3 + 4x^2+3x+5\quad\text{with}\quad
f' = \tsfrac{\dd}{\dd x} f(x) = 3x^2 + 8x+3.
\end{displaymath}
The coefficients are from a \emph{commutative} field
$\field{K}$ (for example the rational $\numQ$, the real $\numR$, or
the complex number field $\numC$),
the (non-commuting) variables from a (finite)
alphabet $\alphabet{X}$, for example $\alphabet{X} = \{ x, y, z \}$.

\medskip
For simplicity we focus here (in this motivation)
on the \emph{free associative algebra}
$R := \freeALG{\field{K}}{\alphabet{X}}$,
aka ``algebra of nc polynomials'', and recall the
properties of a (partial) \emph{derivation} $\ncpd_x : R \to R$
(for a fixed $x \in \alphabet{X}$), namely
\begin{itemize}
\item $\ncpd_x (\alpha) = 0$ for $\alpha \in \field{K}$, and more general,
  $\ncpd_x(g) =0$ for $g \in \freeALG{\field{K}}{\alphabet{X}\!\setminus\! \{x\}}$,
\item $\ncpd_x (x) = 1$, and
\item $\ncpd_x (fg) = \ncpd_x(f)\,g + f\, \ncpd_x(g)$.
\end{itemize}
In \cite{Rota1980a}
\ this is called \emph{Hausdorff derivative}.
For a more general (module theoretic) context we refer to 
\cite[Section~2.7]{Cohn2003b}
\ or \cite{Bergman1975b}
.

\medskip
Before we continue we should clarify the wording: To avoid
confusion we refer to the \emph{linear operator} $\ncpd_x$ (for
some $x \in \alphabet{X}$) as \emph{free partial derivation} (or just
\emph{derivation}) and call $\ncpd_x f \in R$ the (free partial)
\emph{derivative} of $f \in R$,
sometimes denoted also as $f_x$ or $f'$ depending on the context.

In the commutative, we usually do not distinguish too much between
algebraic and analytic concepts. But in the (free) non-commutative
setting, analysis is quite subtle \cite{Kaliuzhnyi2014a}
. There are even concepts like ``matrix convexity''
and symbolic procedures to determine (nc) convexity
\cite{Camino2003a}
. However, a systematic treatment of the underlying algebraic
tools was not available so far. We are going to close this gap in the following.

\medskip
After a brief description of the setup
(in particular that of \emph{linear representations} of elements
in free fields) in Section~\ref{sec:fd.intro},
we develop the formalism for \emph{free derivations}
in Section~\ref{sec:fd.der} with the main
result, Theorem~\ref{thr:fd.ncpd}.
To be able to state a (partial) ``free'' chain rule
(in Proposition~\ref{pro:fd.chainrule})
we derive a language for the
\emph{free composition} (and illustrate
how to ``reverse'' it) in Section~\ref{sec:fd.comp}.
And finally, in Section~\ref{sec:fd.newton},
we show how to develop a meta algorithm
``nc Newton'' to find matrix-valued roots
of a non-commutative rational equation.

\medskip
\begin{notation}
The set of the natural numbers is denoted by $\numN = \{ 1,2,\ldots \}$,
that including zero by $\numN_0$.
Zero entries in matrices are usually replaced by (lower) dots
to emphasize the structure of the non-zero entries
unless they result from transformations where there
were possibly non-zero entries before.
We denote by $I_n$ the identity matrix (of size $n$)
respectively $I$ if the size is clear from the context.
By $v\trp$ we denote the transpose of a vector $v$.
\end{notation}

\section{Getting Started}\label{sec:fd.intro}

We represent elements (in free fields) by
\emph{admissible linear systems} (Definition~\ref{def:fd.als}),
which are just a special form of \emph{linear representations}
(Definition~\ref{def:fd.rep}) and ``general'' \emph{admissible systems}
\cite[Section~7.1]{Cohn2006a}
. Rational operations
(scalar multiplication, addition, multiplication, inverse) can be
easily formulated in terms of linear representations
\cite[Section~1]{Cohn1999a}
. For the formulation on the level of admissible linear systems
and the ``minimal'' inverse we refer to
\cite[Proposition~1.13]{Schrempf2017a9}
\ resp.~\cite[Theorem~4.13]{Schrempf2017a9}
.

Let $\field{K}$ be a \emph{commutative} field,
$\aclo{\field{K}}$ its algebraic closure and
$\alphabet{X} = \{ x_1, x_2, \ldots, x_d\}$ be a \emph{finite} (non-empty) alphabet.
$\freeALG{\field{K}}{\alphabet{X}}$ denotes the \emph{free associative
algebra} (or \emph{free $\field{K}$-algebra})
and $\field{F} = \freeFLD{\field{K}}{\alphabet{\alphabet{X}}}$ its \emph{universal field of
fractions} (or ``free field'') 
\cite{Cohn1995a}
,
\cite{Cohn1999a}
. An element in $\freeALG{\field{K}}{\alphabet{X}}$ is called (non-commutative or nc)
\emph{polynomial}.
In our examples the alphabet is usually $\alphabet{X}=\{x,y,z\}$.
Including the algebra of \emph{nc rational series}
\cite{Berstel2011a}
\ we have the following chain of inclusions:
\begin{displaymath}
\field{K}\subsetneq \freeALG{\field{K}}{\alphabet{X}}
  \subsetneq \field{K}^{\text{rat}}\langle\!\langle \alphabet{X}\rangle\!\rangle
  \subsetneq \freeFLD{\field{K}}{\alphabet{X}} =: \field{F}.
\end{displaymath}

\begin{definition}[Inner Rank, Full Matrix \protect{%
\cite[Section~0.1]{Cohn2006a}
}, \cite{Cohn1999a}
]\label{def:fd.full}
Given a matrix $A \in \freeALG{\field{K}}{\alphabet{X}}^{n \times n}$,
the \emph{inner rank}
of $A$ is the smallest number $k\in \numN$
such that there exists a factorization $A = C D$ with
$C \in \freeALG{\field{K}}{\alphabet{X}}^{n \times k}$ and
$D \in \freeALG{\field{K}}{\alphabet{X}}^{k \times n}$.
The matrix $A$ is called \emph{full} if $k = n$,
\emph{non-full} otherwise.
\end{definition}

\begin{theorem}[\protect{%
\cite[Special case of Corollary~7.5.14]{Cohn2006a}
}]\label{thr:fd.uff}
Let $\alphabet{X}$ be an alphabet and $\field{K}$ a commutative field.
The free associative algebra $R = \freeALG{\field{K}}{\alphabet{X}}$
has a universal field of fractions $\field{F}=\freeFLD{\field{K}}{\alphabet{X}}$
such that every full matrix over $R$ can be inverted over $\field{F}$.
\end{theorem}

\begin{remark}
Non-full matrices become singular under a homomorphism into some
field \cite[Chapter~7]{Cohn2006a}
.
In general (rings), neither do full matrices need to be invertible,
nor do invertible matrices need to be full. An example for the former
is the matrix
\begin{displaymath}
B =
\begin{bmatrix}
. & z & -y \\
-z & . & x \\
y & -x & . 
\end{bmatrix}
\end{displaymath}
over the \emph{commutative} polynomial ring $\field{K}[x,y,z]$
which is \emph{not} a Sylvester domain 
\cite[Section~4]{Cohn1989b}
. An example for the latter are rings
without \emph{unbounded generating number} (UGN)
\cite[Section~7.3]{Cohn2006a}
.
\end{remark}

\begin{definition}[Linear Representations, Dimension, Rank
\cite{Cohn1994a,Cohn1999a}
]\label{def:fd.rep}
Let $f \in \field{F}$.
A \emph{linear representation} of $f$ is a triple $\pi_f = (u,A,v)$ with
$u\trp, v \in \field{K}^{n \times 1}$, full
$A = A_0 \otimes 1 + A_1 \otimes x_1 + \ldots + A_d \otimes x_d$
with $A_\ell \in \field{K}^{n\times n}$ for all $\ell \in \{ 0, 1, \ldots, d \}$
and $f = u A\inv v$.
The \emph{dimension} of $\pi_f$ is $\dim \, (u,A,v) = n$.
It is called \emph{minimal} if $A$ has the smallest possible dimension
among all linear representations of $f$.
The ``empty'' representation $\pi = (,,)$ is
the minimal one of $0 \in \field{F}$ with $\dim \pi = 0$.
Let $f \in \field{F}$ and $\pi$ be a \emph{minimal}
linear representation of $f$.
Then the \emph{rank} of $f$ is
defined as $\rank f = \dim \pi$.
\end{definition}

\begin{definition}[Left and Right Families
\cite{Cohn1994a}
]\label{def:fd.family}
Let $\pi=(u,A,v)$ be a linear representation of $f \in \field{F}$
of dimension $n$.
The families $( s_1, s_2, \ldots, s_n )\subseteq \field{F}$
with $s_i = (A\inv v)_i$
and $( t_1, t_2, \ldots, t_n )\subseteq \field{F}$
with $t_j = (u A\inv)_j$
are called \emph{left family} and \emph{right family} respectively.
$L(\pi) = \linsp \{ s_1, s_2, \ldots, s_n \}$ and
$R(\pi) = \linsp \{ t_1, t_2, \ldots, t_n \}$
denote their linear spans (over $\field{K}$).
\end{definition}

\begin{proposition}[\protect{%
\cite[Proposition~4.7]{Cohn1994a}
}]\label{pro:fd.cohn94.47}
A representation $\pi=(u,A,v)$ of an element $f \in \field{F}$
is minimal if and only if both, the left family
and the right family are $\field{K}$-linearly independent.
In this case, $L(\pi)$ and $R(\pi)$ depend only on $f$.
\end{proposition}

\begin{remark}
The left family $(A\inv v)_i$ (respectively the right family $(u A\inv)_j$)
and the solution vector $s$ of $As = v$ (respectively $t$ of $u = tA$)
are used synonymously.
\end{remark}

\begin{definition}[Admissible Linear Systems, Admissible Transformations
\cite{Schrempf2017a9}
]\label{def:fd.als}
A linear representation $\als{A} = (u,A,v)$ of $f \in \field{F}$
is called \emph{admissible linear system} (ALS) for $f$,
written also as $A s = v$,
if $u=e_1=[1,0,\ldots,0]$.
The element $f$ is then the first component
of the (unique) solution vector $s$.
Given a linear representation $\als{A} = (u,A,v)$
of dimension $n$ of $f \in \field{F}$
and invertible matrices $P,Q \in \field{K}^{n\times n}$,
the transformed $P\als{A}Q = (uQ, PAQ, Pv)$ is
again a linear representation (of $f$).
If $\als{A}$ is an ALS,
the transformation $(P,Q)$ is called
\emph{admissible} if the first row of $Q$ is $e_1 = [1,0,\ldots,0]$.
\end{definition}

\section{Free Derivation}\label{sec:fd.der}

Before we define the concrete (partial) \emph{derivation}
and (partial) \emph{directional derivation}, we start with
a (partial) \emph{formal derivation}
on the level of admissible linear systems
and show the basic properties with respect to the
represented elements, in particular that
the (formal) derivation does \emph{not} depend on the ALS
in Corollary~\ref{cor:fd.forder}.

In other words: Given some letter $x \in \alphabet{X}$
and an \emph{admissible linear system} $\als{A}=(u,A,v)$,
there is an \emph{algorithmic} point of view
in which the (free) derivation $\ncpd_x$ defines the
ALS $\als{A}' = \ncpd_x \als{A}$. (Alternatively one
can identify $x$ by its index $\ell \in \{ 1, 2, \ldots, d \}$
and write $\als{A}' = \ncpd_{\ell} \als{A}$.)
Written in a sloppy way, we show that
$\ncpd_x(\als{A} + \als{B}) = \ncpd_x \als{A} + \ncpd_x \als{B}$ and
$\ncpd_x(\als{A} \cdot \als{B})
  = \ncpd_x \als{A} \cdot \als{B} + \als{A} \cdot \ncpd_x \als{B}$,
yielding immediately the \emph{algebraic} point of view
(summarized in Theorem~\ref{thr:fd.ncpd})
by taking the respective first component of the
\emph{unique} solution vectors $f = (s_{\als{A}})_1$
and $g = (s_{\als{B}})_1$.

\begin{remark}
The following definition is much more general then usually needed.
One gets the ``classical'' (partial) derivation with respect to
some letter $x = x_{\ell} \in \alphabet{X}$
(with $\ell \in \{ 1, 2, \ldots, d \}$) for $k=0$
resp.~(the empty word) $a = 1 \in \alphabet{X}^*$.
\end{remark}

\begin{definition}[Formal Derivative]
\label{def:fd.alsfor}
Let $\als{A}=(u,A,v)$ be an \emph{admissible linear system}
of dimension $n \ge 1$ for some element in the free field
$\field{F} = \freeFLD{\field{K}}{\alphabet{X}}$
and $\ell \neq k \in \{ 0, 1, \ldots, d \}$.
The ALS 
\begin{displaymath}
\ncpd_{\ell|k} \als{A} = \ncpd_{\ell|k} (u,A,v) =
\left(
\begin{bmatrix}
u & . 
\end{bmatrix},
\begin{bmatrix}
A & A_{\ell}\otimes x_k \\
. & A
\end{bmatrix},
\begin{bmatrix}
. \\ v
\end{bmatrix}
\right)
\end{displaymath}
(of dimension $2n$)
is called (partial) \emph{formal derivative} of $\als{A}$,
(with respect to $x_{\ell},x_k \in \{ 1 \} \cup \alphabet{X}$).
For $x, a \in \{ 1 \} \cup \alphabet{X}= \{ 1, x_1, x_2, \ldots, x_d \}$
with $x \neq a$ we write also
$\ncpd_{x|a} \als{A}$, having the indices $\ell \neq k \in \{ 0, 1, \ldots, d \}$
of $x$ resp.~$a$ in mind.
\end{definition}

\begin{lemma}\label{lem:fd.alstoele}
Let $\als{A}_f = (u_f,A_f,v_f)$ and $\als{A}_g = (u_g,A_g,v_g)$
be admissible linear systems of dimension $\dim \als{A}_f \ge 1$
resp.~$\dim \als{A}_g \ge 1$.
Fix $x, a \in \{ 1 \} \cup \alphabet{X}$
such that $x \neq a$. Then
$\ncpd_{x|a}(\als{A}_f + \als{A}_g) = \ncpd_{x|a}\als{A}_f + \ncpd_{x|a}\als{A}_g$.
\end{lemma}

\begin{proof}
Let $\ell \neq k \in \{ 0, 1, \ldots, d \}$ be the indices of
$x$ resp.~$a$. We write $A_f^{\ell}$ for the coefficient matrix
$A_{\ell}$ of $A_f$ resp.~$A_g^{\ell}$ for $A_{\ell}$ of $A_g$.
Taking the sum from 
\cite[Proposition~1.13]{Schrempf2017a9}
\ we have
\begin{align*}
\ncpd_{x|a} & (\als{A}_f + \als{A}_g) = \left(
  \begin{bmatrix}
  u_f & . 
  \end{bmatrix},
  \begin{bmatrix}
  A_f & -A_f u_f\trp u_g \\
  . & A_g
  \end{bmatrix}, 
  \begin{bmatrix}
  v_f \\ v_g
  \end{bmatrix}
  \right) \\
&= \left(
  \begin{bmatrix}
  u_f & . & . & .
  \end{bmatrix},
  \begin{bmatrix}
  A_f & -A_f u_f\trp u_g & A_f^{\ell}\otimes a & -A_f^{\ell} u_f\trp u_g \otimes a \\
  . & A_g & . & A_g^{\ell}\otimes a \\
  . & . & A_f & -A_f u_f\trp u_g \\
  . & . & . & A_g
  \end{bmatrix}, 
  \begin{bmatrix}
  . \\ . \\ v_f \\ v_g
  \end{bmatrix}
  \right) \\
&= \left(
  \begin{bmatrix}
  u_f & . & . & .
  \end{bmatrix},
  \begin{bmatrix}
  A_f & A_f^{\ell}\otimes a & -A_f u_f\trp u_g & -A_f^{\ell} u_f\trp u_g \otimes a \\
  . & A_f & .  & -A_f u_f\trp u_g \\
  . & . & A_g & A_g^{\ell}\otimes a \\
  . & . & . & A_g
  \end{bmatrix}, 
  \begin{bmatrix}
  . \\ v_f \\ . \\ v_g
  \end{bmatrix}
  \right) \\
&= \left(
  \begin{bmatrix}
  u_f & . & . & .
  \end{bmatrix},
  \begin{bmatrix}
  A_f & A_f^{\ell}\otimes a & -A_f u_f\trp u_g & 0\\
  . & A_f & .  & 0  \\
  . & . & A_g & A_g^{\ell}\otimes a \\
  . & . & . & A_g
  \end{bmatrix}, 
  \begin{bmatrix}
  . \\ v_f \\ . \\ v_g
  \end{bmatrix}
  \right) \\
&= \left(
  \begin{bmatrix}
  u_f & . 
  \end{bmatrix},
  \begin{bmatrix}
  A_f & A_f^{\ell}\otimes a \\
  . & A_f
  \end{bmatrix},
  \begin{bmatrix}
  . \\ v_f
  \end{bmatrix}
  \right)
  + \left(
  \begin{bmatrix}
  u_g & . 
  \end{bmatrix},
  \begin{bmatrix}
  A_g & A_g^{\ell}\otimes a \\
  . & A_g
  \end{bmatrix},
  \begin{bmatrix}
  . \\ v_g
  \end{bmatrix}
  \right) \\
&= \ncpd_{x|a}\als{A}_f + \ncpd_{x|a}\als{A}_g.
\end{align*}
The two main steps are swapping block rows~2 and~3 and
block columns~2 and~3,
and eliminating the \emph{single} non-zero (first) column in
$-A_f^{\ell} u_f\trp u_g \otimes a$ and $-A_f u_f\trp u_g$
(in block column~4)
using the first column in block column~2.
\end{proof}

\begin{corollary}\label{cor:fd.forder}
Let $f,g \in \field{F}$ be given by the admissible linear systems
$\als{A}_f = (u_f,A_f,v_f)$ and $\als{A}_g = (u_g,A_g,v_g)$ of
dimensions $n_f,n_g \ge 1$ respectively.
Fix $x, a \in \{ 1 \} \cup \alphabet{X}$ such that $x\neq a$.
Then $f=g$ implies that $\ncpd_{x|a}\als{A}_f - \ncpd_{x|a}\als{A}_g$
is an ALS for $0 \in \field{F}$.
\end{corollary}

\begin{definition}[Formal Derivative]
\label{def:fd.forder}
Let $f \in \field{F}$ be given by the ALS $\als{A}=(u,A,v)$
and fix $x,a \in \{ 1 \} \cup \alphabet{X} = \{ 1, x_1, x_2, \ldots, x_d \}$
such that $x \neq a$.
Denote by $\ncpd_{x|a} f$ the element defined by the ALS $\ncpd_{x|a} \als{A}$.
The map $\ncpd_{x|a} : \field{F} \to \field{F}$,
$f \mapsto \ncpd_{x|a} f$ is called
(partial) \emph{formal derivation},
the element $\ncpd_{x|a} f$ (partial) \emph{formal derivative} of $f$.
\end{definition}

\begin{corollary}\label{cor:fd.linear}
For each $x,a \in \{ 1 \} \cup \alphabet{X}$
with $x \neq a$,
the formal derivation $\ncpd_{x|a} : \field{F} \to \field{F}$
is a \emph{linear} map.
\end{corollary}

Now we are almost done.
Before we show the \emph{product rule}
in the following Lemma~\ref{lem:fd.product},
we have a look into the left family of the ALS of the
(formal) derivative of a polynomial. Let $p = x^3 \in \field{F}$
(and $a=1$).
A (minimal) ALS for $\ncpd_{x|1} p$ is given by
\begin{displaymath}
\begin{bmatrix}
1 & -x & . & . & 0 & -1 & . & . \\
. & 1 & -x & . & . & 0 & -1 & . \\
. & . & 1 & -x & . & . & 0 & -1 \\
. & . & . & 1 & . & . & . & 0 \\
. & . & . & . & 1 & -x & . & . \\
. & . & . & . & . & 1 & -x & . \\
. & . & . & . & . & . & 1 & -x \\
. & . & . & . & . & . & . & 1
\end{bmatrix}
s = 
\begin{bmatrix}
0 \\ 0 \\ 0 \\ 0 \\ . \\ . \\ . \\ 1
\end{bmatrix},
\quad
s =
\begin{bmatrix}
3 x^2 \\ 2x \\ 1 \\ 0 \\ x^3 \\ x^2 \\ x \\ 1
\end{bmatrix}.
\end{displaymath}
Notice that the first four entries in the left family of $\ncpd_{x|1} \als{A}$
are $s_i = \ncpd_{x|1} s_{i+4}$.

\begin{lemma}[Product Rule]\label{lem:fd.product}
Let $f,g \in \field{F}$ be given by the
admissible linear systems $\als{A}_f = (u_f, A_f, v_f)$
and $\als{A}_g = (u_g, A_g, v_g)$ of dimension $n_f,n_g \ge 1$
respectively.
Fix $x \in \alphabet{X}$ and $a \in \{ 1 \} \cup \alphabet{X} \setminus \{ x \}$.
Then $\ncpd_{x|a} (fg) = \ncpd_{x|a} f \,g + f \, \ncpd_{x|a} g$.
\end{lemma}

\begin{proof}
Let $\ell \neq k \in \{ 0, 1, \ldots, d \}$ be the indices of
$x$ resp.~$a$. We write $A_f^{\ell}$ for the coefficient matrix
$A_{\ell}$ of $A_f$ resp.~\raisebox{0pt}[0pt][0pt]{$A_g^{\ell}$}
for $A_{\ell}$ of $A_g$.
We take the sum and the product from
\cite[Proposition~1.13]{Schrempf2017a9}
\ and start with the ALS from the right hand side,
\begin{displaymath}
\begin{bmatrix}
A_f & A_f^{\ell}\otimes a & 0 & -A_f u_f \trp u_g  & . & . \\
. & A_f & -v_f u_g & . & . & . \\
. & . & A_g & . & . & . \\
. & . & . & A_f & -v_f u_g & . \\
. & . & . & . & A_g & A_g^{\ell}\otimes a \\
. & . & . & . & . & A_g 
\end{bmatrix}
s =
\begin{bmatrix}
. \\ . \\ v_g \\ . \\ . \\ v_g
\end{bmatrix},
\end{displaymath}
subtract block row~6 from block row~3, add block column~3 to block
column~6 and remove block row/column~3 to get the ALS
\begin{displaymath}
\begin{bmatrix}
A_f & A_f^{\ell}\otimes a & -A_f u_f \trp u_g  & . & . \\
. & A_f & . & . & -v_f u_g \\
. & . & A_f & -v_f u_g & . \\
. & . & . & A_g & A_g^{\ell}\otimes a \\
. & . & . & . & A_g 
\end{bmatrix}
s =
\begin{bmatrix}
. \\ . \\ . \\ . \\ v_g
\end{bmatrix}.
\end{displaymath}
Now we can add block row~3 to block row~1 and
eliminate the remaining columns in block $(1,3)$ by
the columns $\{ 2, 3, \ldots, n_f \}$ from block $(1,1)$, remove
block row/column~3 to get the ALS
\begin{displaymath}
\begin{bmatrix}
A_f & A_f^{\ell}\otimes a & -v_f u_g & . \\
. & A_f & . & -v_f u_g \\
. & . & A_g & A_g^{\ell}\otimes a \\
. & . & . & A_g 
\end{bmatrix}
s =
\begin{bmatrix}
. \\ . \\ . \\ v_g
\end{bmatrix}.
\end{displaymath}
Swapping block rows~2 and~3 and block columns~2 and~3 yields the ALS
\begin{displaymath}
\begin{bmatrix}
A_f & -v_f u_g & A_f^{\ell}\otimes a & 0 \\
. & A_g & . & A_g^{\ell}\otimes a \\
. & . & A_f & -v_f u_g \\
. & . & . & A_g
\end{bmatrix}
s =
\begin{bmatrix}
. \\ . \\ . \\ v_g
\end{bmatrix}
\end{displaymath}
of the left hand side $\ncpd_{x|0} (fg)$.
Notice the upper right zero in the system matrix
which is because of $x \neq 1$.
\end{proof}

\begin{definition}[Partial Derivative]\label{def:fd.ncpd}
Let $f \in \field{F}$, $x \in \alphabet{X}$ and $a \in \alphabet{X} \setminus \{ x\}$.
The element $\ncpd_{x} f := \ncpd_{x|1} f$ is called
\emph{partial derivative} of $f$.
The element $\ncpd_{x|a} f$ is called
(partial) \emph{directional derivative} of $f$ (with respect to $a$).
\end{definition}

\begin{theorem}[Free Derivation]\label{thr:fd.ncpd}
Let $x \in \alphabet{X}$. Then the (partial) \emph{free derivation}
$\ncpd_x: \field{F} \to \field{F} = \freeFLD{\field{K}}{\alphabet{X}}$
is the \emph{unique} map with the properties
\begin{itemize}
\item $\ncpd_x h = 0$ for all
  $h \in \freeFLD{\field{K}}{\alphabet{X}\!\setminus\!\{ x \}}$,
\item $\ncpd_x x = 1$, and
\item $\ncpd_x (fg) = \ncpd_x f \, g + f\, \ncpd_x g$ for all
  $f,g \in \field{F} = \freeFLD{\field{K}}{\alphabet{X}}$.
\end{itemize}
\end{theorem}

\begin{proof}
Let $h$ be given by the ALS $\als{A} = (u,A,v)$ and let
$\ell \in \{ 1, 2, \ldots, d \}$ such that $x = x_{\ell}$.
We just need to recall
the ALS for $\ncpd_x h$,
\begin{displaymath}
\begin{bmatrix}
A & A_{\ell}\otimes 1 \\
. & A
\end{bmatrix}
\begin{bmatrix}
s' \\ s''
\end{bmatrix}
=
\begin{bmatrix}
. \\ v
\end{bmatrix}
\end{displaymath}
and observe that $A_{\ell} = 0$ and thus $As'=0$,
in particular the first component of $s'$.
Therefore $\ncpd_x h = 0$.
For $\ncpd_x x = 1$ we need to minimize
\begin{displaymath}
\begin{bmatrix}
1 & -x & . & -1 \\
. & 1 & . & . \\
. & . & 1 & -x \\
. & . & . & 1 
\end{bmatrix}
s = 
\begin{bmatrix}
. \\ . \\ . \\ 1
\end{bmatrix}.
\end{displaymath}
And the product rule
$\ncpd_x (fg) = \ncpd_x f \, g + f\, \ncpd_x g$
is due to Lemma~\ref{lem:fd.product}.
For the uniqueness we assume that there exists another
$\ncpd'_x : \field{F} \to \field{F}$ with the same properties.
From the product rule we obtain
$\ncpd_x(x f) = f + x\, \ncpd_x f = f + x\, \ncpd'_x f = \ncpd'_x(x f)$,
that is, $x(\ncpd_x f - \ncpd'_x f) = 0$ for \emph{all} $f \in \field{F}$,
thus $\ncpd_x = \ncpd'_x$.
\end{proof}

\begin{corollary}[Hausdorff Derivation
\cite{Rota1980a}
]\label{cor:fd.hausdorff}
Let $x \in \alphabet{X}$. Then
$\ncpd_x \kappa =0$ for all $\kappa \in \field{K}$,
$\ncpd_x y=0$ for all $y \in \alphabet{X} \setminus \{ x \}$,
$\ncpd_x x = 1$, and
$\ncpd_x(fg) = \ncpd_x f\, g + f\, \ncpd_x g$
for all $f,g \in \freeALG{\field{K}}{\alphabet{X}}$.
\end{corollary}

\begin{remark}
More general
\cite[Theorem~7.5.17]{Cohn2006a}
: ``Any derivation of a Sylvester domain extends to a derivation
of its universal field of fractions.''
Recall however that the \emph{cyclic derivative}
is \emph{not} (from) a derivation
\cite[Section~1]{Rota1980a}
. For a discussion of cyclic derivatives of nc algebraic power series
we refer to \cite{Reutenauer1983a}
.
\end{remark}

\begin{proposition}
Let $f \in \field{F}$, $x,y \in \alphabet{X}$
and $a,b \in \alphabet{X} \setminus \{ x, y \}$
with $a=b$ if and only if $x=y$.
Then $\ncpd_{x|a} (\ncpd_{y|b} f) = \ncpd_{y|b} (\ncpd_{x|a} f)$,
that is, the (partial) derivations $\ncpd_{x|a}$ and $\ncpd_{y|b}$
\emph{commute}.
\end{proposition}

\begin{proof}
Let $l,k \in \{ 1, 2, \ldots, d \}$ the indices of $x$ resp.~$y$.
There is nothing to show for the trivial case $x=y$,
thus we can assume $l \neq k$.
Let $f$ be given by the admissible linear system $\als{A} = (u,A,v)$.
Then the ``left'' ALS $\ncpd_{x|a} (\ncpd_{y|b} \als{A})$ is
\begin{displaymath}
\begin{bmatrix}
A & A_k \otimes b & A_l \otimes a & 0 \\
. & A & 0 & A_l \otimes a \\
. & . & A & A_k \otimes b \\
. & . & . & A
\end{bmatrix}
s = 
\begin{bmatrix}
. \\ . \\ . \\ v
\end{bmatrix}.
\end{displaymath}
Swapping block rows/columns~2 and~3 yields
the ``right'' ALS $\ncpd_{y|b}(\ncpd_{x|a} \als{A})$:
\begin{displaymath}
\begin{bmatrix}
A & A_l \otimes a & A_k \otimes b  & 0 \\
. & A & 0 & A_k \otimes b \\
. & . & A & A_l \otimes a \\
. & . & . & A
\end{bmatrix}
s = 
\begin{bmatrix}
. \\ . \\ . \\ v
\end{bmatrix}.
\end{displaymath}
\end{proof}

Since there is no danger of ambiguity, we can
define the (free) ``higher'' derivative of $f \in \field{F}$
as $\ncpd_w f$ for each word $w$ in the \emph{free monoid}
$\alphabet{X}^*$ with the ``trivial'' derivative
$\ncpd_{(1)} f = f$.
Let $w \in \alphabet{X}^*$ and $\sigma(w)$ denote
any permutation of the letters of $w$.
Then $\ncpd_w f = \ncpd_{\sigma(w)} f$.

The proof for nc formal power series in 
\cite[Proposition~1.8]{Popescu2006a}
\ is based on words (monomials), that is,
$\ncpd_x (\ncpd_y w ) = \ncpd_y (\ncpd_x w) \in \freeALG{\field{K}}{\alphabet{X}}$.
Recall that one gets the (nc) \emph{rational} series by
intersecting the (nc) series and the \emph{free field}
\cite[Section~9]{Reutenauer2008a}
:
\begin{displaymath}
\field{K}\langle\!\langle \alphabet{X}\rangle\!\rangle
  \cap \freeFLD{\field{K}}{\alphabet{X}}
= \field{K}^{\text{rat}}\langle\!\langle \alphabet{X}\rangle\!\rangle.
\end{displaymath}

Overall, (free) nc derivation does not appear that often
in the literature. And when there is some discussion
it is (almost) always connected with ``not simple''
\cite[Section~1]{Rota1980a}
, ``complicated'' 
\cite[Section~14.3]{Hackbusch2009a}
, etc. This is however \emph{not} due to the Hausdorff derivation
but to the use of (finite) formal series as representation
(for nc polynomials). Using \emph{linear representations}
in the sense of Cohn and Reutenauer 
\cite{Cohn1994a}
\ for elements in the free field $\field{F}$ can even
reveal additional structure, as indicated in Example~\ref{ex:fd.matfact}
(below).
For the somewhat more ``complicated'' example
\begin{align*}
\tilde{p}
  &= 3 c y x b + 3 x b y x b + 2 c y x a x + c y b x b - c y a x b
      - 2 x b y x a x + 4 x b y b x b\\
  &\quad - 3 x b y a x b + 3 x a x y x b  - 3 b x b y x b + 6 a x b y x b
      + 2 x a x y x a x + x a x y b x b\\
  &\quad - x a x y a x b  - 2 b x b y x a x  - b x b y b x b + b x b y a x b
      + 5 a x b y b x b - 4 a x b y a x b
\end{align*}
from \cite[Section~8.2]{Camino2006a}
\ we refer to \cite[Example~3.7]{Schrempf2019a}
.

\begin{Example}\label{ex:fd.matfact}
Let $p = xyzx$. A (minimal) polynomial ALS for $p$ is
\begin{displaymath}
\begin{bmatrix}
1 & -x & . & . & . \\
. & 1 & -y & . & . \\
. & . & 1 & -z & . \\
. & . & . & 1 & -x \\
. & . & . & . & 1
\end{bmatrix}
s=
\begin{bmatrix}
. \\ . \\ . \\ . \\ 1
\end{bmatrix}.
\end{displaymath}
Then $\ncpd_x p = xyz + yzx$ admits a factorization into matrices
\cite[Section~3]{Schrempf2019a}
:
\begin{displaymath}
\ncpd_x p =
\begin{bmatrix}
x & y
\end{bmatrix}
\begin{bmatrix}
y & . \\
. & z
\end{bmatrix}
\begin{bmatrix}
z \\ x
\end{bmatrix}.
\end{displaymath}
A \emph{minimal} ALS for $\ncpd_x p$ is
\begin{displaymath}
\begin{bmatrix}
1 & -x & -y & . & . & . \\
. & 1 & . & -y & 0 & . \\
. & . & 1 & 0 & -z & . \\
. & . & . & 1 & . & -z \\
. & . & . & . & 1 & -x \\
. & . & . & . & . & 1
\end{bmatrix}
s =
\begin{bmatrix}
. \\ . \\ . \\ . \\ . \\ 1
\end{bmatrix}.
\end{displaymath}
\end{Example}

\medskip
\begin{example}
We will use the \emph{directional derivative}
later in Section~\ref{sec:fd.newton} for the (nc)
Newton iteration. For $q = x^2$ we get
the ``Sylvester equation'' $\ncpd_{x|a} q = xa + ax$
which is \emph{linear} in $a$.
\end{example}

\medskip
In (the next) Section~\ref{sec:fd.comp}
we have a look on the
``free'' chain rule which will turn out to be
very elegant.
We avoid the term ``function'' here since one needs
to be careful with respect to evaluation
(domain of definition), e.g.~$f=(xy-yx)\inv$ is
not defined for diagonal matrices.
For the efficient evaluation of polynomials (by matrices)
one can use \emph{Horner Systems}
\cite{Schrempf2019a}
. A generalization to elements in free fields
is considered in future work.
If there is a ``compositional structure'' available
(in admissible linear systems) it could be used to
further optimize evaluation.

\begin{remark}
For details on minimization (of linear representations)
we refer to \cite{Schrempf2018a9}
. Notice in particular that the construction
(of the formal derivative) in Definition~\ref{def:fd.alsfor}
\emph{preserves} refined pivot blocks.
Therefore, if $\als{A}$ is refined,
\emph{linear} (algebraic) techniques
suffice for minimization of $\ncpd_x \als{A}$.
\end{remark}

\medskip
Last but not least, given the alphabet
$\alphabet{X} = \{ x_1, x_2, \ldots, x_d \}$,
we can define the ``free'' (canonical) \emph{gradient}
$\Grad f = [\ncpd_1 f, \ncpd_2 f, \ldots, \ncpd_d f]\trp
= [\ncpd_{x_1} f, \ncpd_{x_2} f, \ldots, \ncpd_{x_d} f]\trp \in \field{F}^d$
for some $f \in \freeFLD{\field{K}}{\alphabet{X}}$.
\emph{Cyclic} gradients are discussed in
\cite{Voiculescu2000b}
.

And for a ``vector valued'' element $\mathbf{f} = (f_1,f_2,\ldots,f_d)$
we can define the \emph{Jacobian matrix}
$J(\mathbf{f}) = (\ncpd_j f_i)_{i,j=1}^d = (\ncpd_{x_j} f_i)_{i,j=1}^d$.
For a discussion of non-commutative Jacobian matrices
on the level of nc formal power series
we refer to \cite{Reutenauer1992a}
.

\newpage
\section{Free Composition}\label{sec:fd.comp}

\begin{figure}
\begin{center}
$\xymatrix@C=0em{  &
{\als{A}_g \sim g \in \field{F}_{\alphabet{Y}}}
    \ar[dl]_{\circ \mathbf{f}}
    \ar[dr]^{\ncpd_{\mathbf{y}|\mathbf{y}'}} \\
{\tilde{\als{A}}_h \sim h \in \field{F}_{\alphabet{X}}}
    \ar[d]_{\text{lin.}} & &
{\als{A}'_g \sim g' \in \field{F}_{\alphabet{Y} \cup \alphabet{Y}'}}
    \ar[d]^{\circ (\mathbf{f},\mathbf{f}')} \\
{\als{A}_h \sim h \in \field{F}_{\alphabet{X}}}
    \ar[dr]_{\ncpd_x } & &
{\tilde{\als{A}}'_h \sim h' \in \field{F}_{\alphabet{X}}}
    \ar[dl]^{\text{lin.}} \\
& {\ncpd_x \als{A}_h \sim \ncpd_x h = h' \sim \als{A}_h'}
}$
\end{center}
\caption{Let $\mathbf{f}=(f_1,\ldots,f_d)$ with
  $f_i \in \field{F}_{\alphabet{X}}=\freeFLD{\field{K}}{\alphabet{X}}$
  given by the $d$-tuple of admissible linear systems
  $\als{A}_{\mathbf{f}} = (\als{A}_1, \ldots, \als{A}_d)$
  and $g \in \field{F}_{\alphabet{Y}}$ given by $\als{A}_g=(u_g,A_g,v_g)$
  such that the system matrix $A_g$ remains \emph{full} when we replace
  each letter $y_i \in \alphabet{Y}$ by the respective element $f_i$,
  written as $A_g \circ \mathbf{f}$.
  Then $h = g \circ \mathbf{f} \in \field{F}_{\alphabet{X}}$ is defined by the
  \emph{admissible system} $\tilde{\als{A}}_h = (u_g, A_g \circ \mathbf{f}, v_g)$
  which we \emph{linearize} to obtain an ALS $\als{A}_h$ for $h$
  before we apply the (partial) derivation $\ncpd_x$ (left path).
  On the other hand, we can apply the (directional) derivation
  $\ncpd_{\mathbf{y}|\mathbf{y}'}$ by going over to the free field
  $\field{F}_{\alphabet{Y} \cup \alphabet{Y}'}
    = \freeFLD{\field{K}}{\alphabet{Y} \cup \alphabet{Y}'}$
  with an extended alphabet (with ``placeholders'' $y_i'$), yielding
  $g' = \ncpd_{y_1|y_1'} g + \ldots + \ncpd_{y_d|y_d'} g$
  given by some ALS $\als{A}'_g=(u'_g,A'_g,v'_g)$.
  Then $h' = g' \circ (\mathbf{f},\mathbf{f}') \in \field{F}_{\alphabet{X}}$
  is defined by the \emph{admissible system}
  $\tilde{\als{A}}_h' = \bigl(u'_g,A'_g \circ (\mathbf{f},\mathbf{f}'),v'_g\bigr)$,
  where also each letter $y_i'\in \alphabet{Y}'$ is replaced by the respective
  element $f_i' = \ncpd_x f_i \in \field{F}_{\alphabet{X}}$.
  After linearization we get an ALS $\als{A}'_h$ such that
  $\ncpd_x \als{A}_h - \als{A}'_h = 0$,
  that is, $\ncpd_x h = h'$ (right path).
  }
\label{fig:fd.chain}
\end{figure}

To be able to formulate a (partial) ``free'' \emph{chain rule} in an elegant way,
we need a suitable notation. It will turn out that Cohn's
\emph{admissible systems} \cite[Section~7.1]{Cohn2006a}
\ provide the perfect framework for the ``expansion'' of
letters by elements from another free field.

First we recall the ``classical'' (analytical) chain rule:
Let $X$, $Y$ and $Z$ be (open) sets, $f=f(x)$ differentiable on $X$,
$g=g(y)$ differentiable on $Y$ and $h=h(x)=g\bigl(f(x)\bigr)$.
Then $\frac{\dd}{\dd x} h = \frac{\dd}{\dd f}g\,\frac{\dd}{\dd x} f$
resp.~$h'(x) = g'\bigl(f(x)\bigr) f'(x)$.
\begin{displaymath}
\xymatrix{
  {X} \ar[r]^{f} \ar@/^-3ex/[rr]_{h}
& {Y} \ar[r]^{g}
& {Z} 
}
\end{displaymath}

\begin{notation}
For a fixed $d \in \numN$ let
$\alphabet{X} = \{ x_1, x_2, \ldots, x_d \}$,
$\alphabet{Y} = \{ y_1, y_2, \ldots, y_d \}$ and
$\alphabet{Y}' = \{ y_1', y_2', \ldots, y_d' \}$
be \emph{pairwise disjoint} alphabets, that is,
$\alphabet{X} \cap \alphabet{Y} =
\alphabet{X} \cap \alphabet{Y}' =
\alphabet{Y} \cap \alphabet{Y}' = \emptyset$.
By $\field{F}_{\alphabet{Z}}$ we denote the free field
$\freeFLD{\field{K}}{\alphabet{Z}}$. 
Let $A \in \freeALG{\field{K}}{\alphabet{Y}}^{n \times n}$ be a \emph{linear full}
matrix and $\mathbf{f}=(f_1,f_2,\ldots,f_d)$ a $d$-tuple of elements
$f_i \in \field{F}_{\alphabet{X}}$. By $A \circ \mathbf{f}$
we denote the (not necessarily full) $n \times n$ matrix over
$\field{F}_{\alphabet{X}}$ where each letter $y_i \in \alphabet{Y}$
is replaced by the corresponding $f_i \in \field{F}_{\alphabet{X}}$,
that is,
\begin{displaymath}
A \circ \mathbf{f}
  = A_0 \otimes 1 + A_1 \otimes f_1 + \ldots + A_d \otimes f_d
\quad\in \field{F}_{\alphabet{X}}^{n \times n}.
\end{displaymath}
We write $A \circ \als{A}_{\mathbf{f}}
  = A \circ (\als{A}_{f_1}, \als{A}_{f_2}, \ldots, \als{A}_{f_d})$
for a \emph{linearized} version induced by the
$d$-tuple of admissible linear systems $\als{A}_{f_i} = (u_{f_i},A_{f_i},v_{f_i})$.
\end{notation}

\medskip
Now let $g \in \field{F}_{\alphabet{Y}}$ be given by the ALS
$\als{A}_g = (u_g,A_g,v_g)$ and
$\mathbf{f}=(f_1,f_2,\ldots,f_d)\in \field{F}_{\alphabet{X}}^d$
such that $A_g \circ \mathbf{f}$ is \emph{full}.
Then (the \emph{unique} element)
$h = g \circ \mathbf{f} \in \field{F}_{\alphabet{X}}$ is
defined by the \emph{admissible system}
$\tilde{\als{A}}_h = (u_g, A_g\circ \mathbf{f}, v_g)$
and we write
\begin{displaymath}
\als{A}_h = (u_h,A_h,v_h) = 
  \bigl(u_g \circ \als{A}_{\mathbf{f}},
        A_g \circ \als{A}_{\mathbf{f}},
        v_g \circ \als{A}_{\mathbf{f}}\bigr)
=: \als{A}_g \circ \als{A}_{\mathbf{f}}
\end{displaymath}
for a \emph{linearized} version using
``linearization by enlargement''
\cite[Section~5.8]{Cohn2006a}
.
(For details we refer to the proof of Proposition~\ref{pro:fd.chainrule} below.)
Fixing some $x \in \alphabet{X}$,
we get the (partial) derivative $\ncpd_x h$ of $h$
via the derivative $\ncpd_x \als{A}_h$
\begin{displaymath}
\begin{bmatrix}
A_h & A_h^x \otimes 1 \\
. & A_h
\end{bmatrix}
s = 
\begin{bmatrix}
. \\ v_h
\end{bmatrix}.
\end{displaymath}
To give a meaning to the right hand side of
$\ncpd_x h = \ncpd_x (g \circ \mathbf{f})$,
we introduce the ``total'' (directional) derivative
$g' := \ncpd_{\mathbf{y}|\mathbf{y}'} g
  = \ncpd_{y_1|y_1'} g + \ncpd_{y_2|y_2'} g + \ldots + \ncpd_{y_d|y_d'} g
  \in\field{F}_{\alphabet{Y} \cup \alphabet{Y}'}$
given by the ALS
\begin{displaymath}
\als{A}_g' :=
\ncpd_{\mathbf{y}|\mathbf{y}'} \als{A}_g = \left(
\begin{bmatrix}
u_g & . 
\end{bmatrix},
\begin{bmatrix}
A_g & \sum_{i=1}^d A_g^{(i)}\otimes y_i' \\
. & A_g
\end{bmatrix},
\begin{bmatrix}
. \\ v_g
\end{bmatrix}
\right)
\end{displaymath}
with letters $y_i' \in \alphabet{Y}'$.
Using a similar notation for the derivatives
\begin{displaymath}
\mathbf{f}' = (f_1',f_2',\ldots,f_d')
  := (\ncpd_x f_1, \ncpd_x f_2, \ldots, \ncpd_x f_d) = \ncpd_x \mathbf{f},
\end{displaymath}
we can write
\begin{equation}\label{eqn:fd.chain}
h' := \ncpd_x(g \circ \mathbf{f}) 
  = g' \circ (\mathbf{f},\mathbf{f}')
  = \ncpd_{\mathbf{y}|\mathbf{y}'} g \circ (\mathbf{f},\mathbf{f}')
\end{equation}
given by the \emph{admissible system}
$\als{A}_g' \circ  (\mathbf{f},\mathbf{f}')
  = \ncpd_{\mathbf{y}|\mathbf{y}'} \als{A}_g \circ (\mathbf{f},\mathbf{f}')$.
After an illustration in the following example,
we show in Proposition~\ref{pro:fd.chainrule} that indeed
\begin{displaymath}
\ncpd_x h = h' = \ncpd_x(g \circ \mathbf{f})
  = \ncpd_{\mathbf{y}|\ncpd_x} g \circ \mathbf{f}
\end{displaymath}
using an abbreviation for the right hand side of \eqref{eqn:fd.chain}.
For an overview see Figure~\ref{fig:fd.chain}.

\medskip
\begin{remark}
Cohn writes an \emph{admissible system} $\als{A} = (u,A,v)$
as ``block'' $[A,v]$ with not necessarily scalar column~$v$
\cite[Section~7.1]{Cohn2006a}
. A ring homomorphism which preserves fullness of matrices is
called \emph{honest} \cite[Section~5.4]{Cohn2006a}
.
\end{remark}

\begin{remark}
It is \emph{crucial} that $A_g \circ \mathbf{f}$ is \emph{full}
to be able to define the composition.
On a purely algebraic level this is sufficient to
define the (partial) chain rule.
For a $d$-tuple $\mathbf{g}=(g_1,g_2,\ldots,g_d)$ given by
admissible linear systems
$\als{A}_{\mathbf{g}}=(\als{A}_{g_1},\als{A}_{g_2},\ldots,\als{A}_{g_d})$
with $\als{A}_{g_i} =(u_{g_i},A_{g_i},v_{g_i})$ we
need $A_{g_i} \circ \mathbf{f}$ \emph{full}
for \emph{all} $i \in \{ 1, 2, \ldots, d \}$.
\end{remark}

\begin{Example}
Here we take $\alphabet{X} = \{ x,y \}$,
$\alphabet{Y} = \{ f, p \}$ and abuse notation.
Let $f = (x\inv +y)\inv \in \field{F}_{\alphabet{X}}$,
$p = xy \in \field{F}_{\alphabet{X}}$ and
$g = pfp \in \field{F}_{\alphabet{Y}}$.
Then $h = g \circ (f,p) \in \field{F}_{\alphabet{X}}$
is given by the (minimal) ALS
\begin{displaymath}
\begin{bmatrix}
1 & -x & . & . & . & . \\
. & 1 & -y & . & . & . \\
. & . & 1 & -x & . & . \\
. & . & y & 1 & -x & . \\
. & . & . & . & 1 & -y \\
. & . & . & . & . & 1
\end{bmatrix}
s = 
\begin{bmatrix}
. \\ . \\ . \\ . \\ . \\ 1
\end{bmatrix},
\end{displaymath}
and $\ncpd_x h = \ncpd_x (pfp) = \ncpd_x p \, fp + p\,\ncpd_x f\, p + pf\,\ncpd_x p$ by
\begin{equation}\label{eqn:fd.chain.1}
\left[\!\!
\begin{array}{cccccc|cccccc}
\mtxstrut
1 & -x &  . &  . &  . &  . &  . & -1 &  . &  . &  . &  . \\
. &  1 & -y &  . &  . &  . &  . &  . &  . &  . &  . &  . \\
. &  . &  1 & -x &  . &  . &  . &  . &  . & -1 &  . &  . \\
. &  . &  y &  1 & -x &  . &  . &  . &  . &  . & -1 &  . \\
. &  . &  . &  . &  1 & -y &  . &  . &  . &  . &  . &  . \\
. &  . &  . &  . &  . &  1 &  . &  . &  . &  . &  . &  . \\\hline\mtxstrut
  &    &    &    &    &    &  1 & -x &  . &  . &  . &  . \\
  &    &    &    &    &    &  . &  1 & -y &  . &  . &  . \\
  &    &    &    &    &    &  . &  . &  1 & -x &  . &  . \\
  &    &    &    &    &    &  . &  . &  y &  1 & -x &  . \\
  &    &    &    &    &    &  . &  . &  . &  . &  1 & -y \\
  &    &    &    &    &    &  . &  . &  . &  . &  . &  1
\end{array}
\!\!\right]
s = 
\left[\!\!
\begin{array}{c}
. \\ . \\ . \\ . \\ . \\ . \\\hline\mtxstrut
. \\ . \\ . \\ . \\ . \\ 1
\end{array}
\!\!\right].
\end{equation}
On the other hand,
$\ncpd_x (\ncpd_f g + \ncpd_p g) = \ncpd_x (\ncpd_f g) + \ncpd_x (\ncpd_p g)$
is given by (the admissible system)
\begin{displaymath}
\begin{bmatrix}
1 & -p & . & . & . & -\ncpd_x p & . & . \\
. & 1 & -f & . & . & . & -\ncpd_x f & . \\
. & . & 1 & -p & . & . & . & -\ncpd_x p \\
. & . & . & 1 & . & . & . & . \\
. & . & . & . & 1 & -p & . & . \\
. & . & . & . & . & 1 & -f & . \\
. & . & . & . & . & . & 1 & -p \\
. & . & . & . & . & . & . & 1 
\end{bmatrix}
s = 
\begin{bmatrix}
. \\. \\ . \\ . \\ . \\ . \\ . \\ 1
\end{bmatrix}.
\end{displaymath}
The summands $\ncpd_x p \, fp$ and $pf\,\ncpd_x p$ are easy
to read off in the ALS~\eqref{eqn:fd.chain.1}.
(Alternatively one could add row~8 to row~1 resp.~row~11 to row~4.)
To read off $p\,\ncpd_x f \, p = - p \, (x\inv + y)^{-2} \,p$,
we just need to recall the (minimal) ALS
\begin{displaymath}
\begin{bmatrix}
1 & -x & . & -1 \\
y & 1 & . & . \\
  &   & 1 & -x \\
  &   & y & 1
\end{bmatrix}
s = 
\begin{bmatrix}
. \\ . \\ . \\ 1
\end{bmatrix}.
\end{displaymath}
In other words and with $g \in \field{F}_{\alphabet{Y}}$
given by the ALS $\als{A}_g = (u_g,A_g,v_g)$ of dimension~$n$: $\ncpd_x h$
is given by the \emph{admissible system}
$\ncpd_{\mathbf{y}|\ncpd_x} \als{A}_g$,
\begin{displaymath}
\begin{bmatrix}
A_g & \ncpd_x A_g \\
. & A_g
\end{bmatrix}
s = 
\begin{bmatrix}
. \\ v
\end{bmatrix}
\end{displaymath}
of dimension~$2n$. Notice that $\ncpd_x A_g$ is understood here
in a purely symbolic way, that is, over $\field{F}_{\alphabet{Y} \cup \alphabet{Y}'}$
with additional letters ``$\ncpd_x y$'' in $\alphabet{Y}'$ for each $y\in \alphabet{Y}$.
In this sense $\ncpd_{\mathbf{y}|\ncpd_x} \als{A}_g$ is actually \emph{linear}.
\end{Example}

\begin{proposition}[Free Chain Rule]\label{pro:fd.chainrule}
Let $\alphabet{X} = \{ x_1, x_2, \ldots, x_d \}$,
$\alphabet{Y} = \{ y_1, y_2, \ldots, y_d \}$ and
$\alphabet{Y}' = \{ y'_1, y'_2, \ldots, y'_d \}$
be \emph{pairwise disjoint} alphabets
and fix $x \in \alphabet{X}$.
For $g \in \freeFLD{\field{K}}{\alphabet{Y}}$
given by the ALS $\als{A}_g = (u_g,A_g,v_g)$
and the $d$-tuple
$\mathbf{f}=(f_1,f_2,\ldots,f_d)\in \freeFLD{\field{K}}{\alphabet{X}}^d$
such that $A_g \circ \mathbf{f}$ is \emph{full}.
Denote by $\mathbf{f}'$ the $d$-duple
$(\ncpd_x f_1, \ncpd_x f_2, \ldots, \ncpd_x f_d)$.
Then
\begin{displaymath}
\ncpd_x (g \circ \mathbf{f}) =
  \ncpd_{\mathbf{y}|\mathbf{y}'} g \circ (\mathbf{f},\mathbf{f}')
  =: \ncpd_{\mathbf{y}|\ncpd_x} g \circ \mathbf{f}.
\end{displaymath}
\end{proposition}

\begin{proof}
For $i \in \{ 1, 2, \ldots, d \}$ let
$f_i \in \field{F}_{\alphabet{X}}$ be given by
the admissible linear systems $\als{A}_{f_i} = (u_{f_i},A_{f_i},v_{f_i})$
respectively. In the following we assume ---without loss of generality---
$d=3$, and decompose $A_g \circ \mathbf{f}$ into
$\bigl[\begin{smallmatrix} A_{11} & A_{12} \\ A_{12} & a \end{smallmatrix}\bigr]$
of size~$n$
with $a = \alpha_0 + \alpha_1 f_1 + \alpha_2 f_2 + \alpha_3 f_3$.
A generalization of the ``linearization by enlargement''
\cite[Section~5.8]{Cohn2006a}
\ is then to start with the \emph{full} matrix
$A_g \circ \mathbf{f} \oplus A_{f_1} \oplus \ldots \oplus A_{f_d}$,
add \raisebox{0pt}[0pt][0pt]{$u_{f_i} A_{f_i}\inv$}
from the corresponding block row to row~$n$, and 
$-\alpha_i A_{f_i}\inv v_{f_i}$ from the corresponding
block column to column~$n$ in the upper left block
(these transformations preserve fullness):
\begin{displaymath}
\begin{bmatrix}
A_{11} & A_{12} & . & . & . \\
A_{21} & \alpha_0 & u_{f_1} & u_{f_2} & u_{f_3} \\
. & -\alpha_1 v_{f_1} & A_{f_1} & . & . \\
. & -\alpha_2 v_{f_2} & . & A_{f_2} & . \\
. & -\alpha_3 v_{f_3} & . & . & A_{f_3}
\end{bmatrix}.
\end{displaymath}
A \emph{partially} linearized system matrix for
$\ncpd_x (\als{A}_g \circ \mathbf{f})$ is 
\begin{equation}\label{eqn:fd.comp.1}
\left[\!\!
\begin{array}{ccccc|ccccc}
\mtxstrut
A_{11} & A_{12} & . & . & . & * & * & . & . & . \\
A_{21} & \alpha_0 & u_{f_1} & u_{f_2} & u_{f_3} & * & 0 & 0 & 0 & 0 \\
. & \makebox[2.2em]{$-\alpha_1 v_{f_1}$} & A_{f_1} & . & .
  & . & 0 & \makebox[2.9em]{$A_{f_1}^x\otimes 1$} & . & . \\
. & \makebox[2.2em]{$-\alpha_2 v_{f_2}$} & . & A_{f_2} & .
  & . & 0 & . & \makebox[2.9em]{$A_{f_2}^x\otimes 1$} & . \\
. & \makebox[2.2em]{$-\alpha_3 v_{f_3}$} & . & . & A_{f_3}
  & . & 0 & . & . & \makebox[2.9em]{$A_{f_3}^x\otimes 1$} \\\hline\mtxstrut
 & & & & & A_{11} & A_{12} & . & . & . \\
 & & & & & A_{21} & \alpha_0 & u_{f_1} & u_{f_2} & u_{f_3} \\
 & & & & & . & \makebox[2.1em]{$-\alpha_1 v_{f_1}$} & A_{f_1} & . & . \\
 & & & & & . & \makebox[2.1em]{$-\alpha_2 v_{f_2}$} & . & A_{f_2} & . \\
 & & & & & . & \makebox[2.1em]{$-\alpha_3 v_{f_3}$} & . & . & A_{f_3}
\end{array}
\!\!\right].
\end{equation}
On the other hand, the system matrix of $\ncpd_{y_i|y_i'} \als{A}_g$ is
\begin{displaymath}
\begin{bmatrix}
A_{11} & A_{12} & A_{11}^{y_i}\otimes y_i' & A_{12}^{y_i}\otimes y_i' \\
A_{21} & a & A_{21}^{y_i}\otimes y_i' & \alpha_i y_i' \\
 & & A_{11} & A_{12} \\
 & & A_{21} & a
\end{bmatrix},
\end{displaymath}
the corresponding \emph{partially} linearized system matrix
of $(\ncpd_{y_i|y_i'} \als{A}_g)\circ (\mathbf{f},f_i')$ is
\begin{displaymath}
\left[\!\!
\begin{array}{ccccc|ccccc}
\mtxstrut
A_{11} & A_{12} & . & . & . & A_{11}^{y_i}\otimes f_i' & A_{12}^{y_i}\otimes f_i' & . & . & . \\
A_{21} & \alpha_0 & u_{f_1} & u_{f_2} & u_{f_3} & A_{21}^{y_i}\otimes f_i' & \fbox{$\alpha_i f_i'$} & . & . & . \\
. & -\alpha_1 v_{f_1} & A_{f_1} & . & . & . & . & . & . & . \\
. & -\alpha_2 v_{f_2} & . & A_{f_2} & . & . & . & . & . & . \\
. & -\alpha_3 v_{f_3} & . & . & A_{f_3} & . & . & . & . & . \\\hline\mtxstrut
 & & & & & A_{11} & A_{12} & . & . & . \\
 & & & & & A_{21} & \alpha_0 & u_{f_1} & u_{f_2} & u_{f_3} \\
 & & & & & . & -\alpha_1 v_{f_1} & A_{f_1} & . & . \\
 & & & & & . & -\alpha_2 v_{f_2} & . & A_{f_2} & . \\
 & & & & & . & -\alpha_3 v_{f_3} & . & . & A_{f_3}
\end{array}
\!\! \right].
\end{displaymath}
To eliminate the boxed entry $\alpha_i f_i'$ we just need to recall
the derivative $\ncpd_x \als{A}_{f_i}$ of $\als{A}_{f_i}=(u_{f_i},A_{f_i},v_{f_i})$,
the invertible matrix $Q$ swaps block columns~1 and~2:
\begin{displaymath}
\left[\!\!
\begin{array}{c|cc}
0 & u_{f_i} & . \\\hline\mtxstrut
. & A_{f_i} & A_{f_i}^x \otimes 1 \\
-\alpha_i v_{f_i} & . & A_{f_i}
\end{array}
\!\!\right]
Q =
\begin{bmatrix}
u_{f_i} & \fbox{$0$} & . \\
A_{f_i} & . & A_{f_i}^x \otimes 1 \\
. & -\alpha_i v_{f_i} & A_{f_i}
\end{bmatrix}.
\end{displaymath}
Thus, after summing up (over $i \in \{ 1, 2, \ldots, d \}$),
we get the system matrix \eqref{eqn:fd.comp.1}
and hence
\begin{displaymath}
\ncpd_x \als{A}_h = \ncpd_x (\als{A}_g \circ \als{A}_f) =
  \ncpd_{\mathbf{y}|\mathbf{y}'} A_{g} \circ \als{A}_{\mathbf{f},\mathbf{f}'},
\end{displaymath}
that is, $\ncpd_x h = \ncpd_x (g \circ \mathbf{f})
  = \ncpd_{\mathbf{y}|\ncpd_x \mathbf{f}} g \circ \mathbf{f}$.
\end{proof}

\newpage
Free (non-commutative) \emph{de}composition is important in
\emph{control theory}
\cite[Section~6.2.2]{deOliveira2006a}
:
``The authors do not know how to fully implement the decompose operation.
Finding decompositions by hand can be facilitated with the use of certain
type of collect command.''

Let $g = f a b f + c f + d e$ and $f = x y + z$,
that is, $f$ solves a Riccati equation.
Given $h = g \circ f = (xy+z)ab(xy+z) + c(xy+z) +de$ by
the (minimal) admissible \emph{linear} system
$\als{A}_h=(u_h,A_h,v_h)$,
\begin{displaymath}
\begin{bmatrix}
1 & -x & -z & . & -d & -c & . & . \\
. & 1 & -y & . & . & . & . & . \\
. & . & 1 & -a & . & . & . & . \\
. & . & . & 1 & . & -b & . & . \\
. & . & . & . & 1 & . & . & -e \\
. & . & . & . & . & 1 & -x & -z \\
. & . & . & . & . & . & 1 & -y \\
. & . & . & . & . & . & . & 1 
\end{bmatrix}
s = 
\begin{bmatrix}
. \\ . \\ . \\ . \\ . \\ . \\ . \\ 1
\end{bmatrix},
\end{displaymath}
one can read off $f = xy+z$ \emph{directly} since
$A_h$ has the form
\begin{displaymath}
\begin{bmatrix}
1 & \fbox{$-f$} & . & -d & -c & . \\
. & 1 & -a & . & . & . \\
. & . & 1 & . & -b & . \\
. & . & . & 1 & . & -e \\
. & . & . & . & 1 & \fbox{$-f$} \\
. & . & . & . & . & 1
\end{bmatrix}
s = 
\begin{bmatrix}
. \\ . \\ . \\ . \\ . \\ 1
\end{bmatrix}.
\end{displaymath}
So, starting with a \emph{minimal} ALS $\tilde{\als{A}}_h$ for $h$
one ``just'' needs to find appropriate (invertible)
transformation matrices $P,Q$ such that $\als{A}_h = P \tilde{\als{A}}_h Q$
using (commutative) Gröbner bases techniques
similar to \cite[Theorem~4.1]{Cohn1999a}
\ or the refinement of pivot blocks
\cite[Section~3]{Schrempf2018a9}
. Although this is quite challenging using brute force
methods, in practical examples it is rather easy using
\emph{free fractions}, that is, \emph{minimal} and \emph{refined}
admissible linear systems, as a ``work bench'' where one
can perform the necessary row and column operations manually.

``The challenge to computer algebra is to start with an expanded version of
$h = g \circ f$, which is a mess that you would likely see in a computer
algebra session, and to automatically motivate the selection of $f$.''
\cite[Section~6.2.1]{deOliveira2006a}

Working with \emph{free fractions} is simple, in particular
with nc polynomials. The main difficulty in working with (finite)
formal power series is that the number of words can grow exponentially
with respect to the \emph{rank}, the minimal dimension of a linear
representation \cite[Table~1]{Schrempf2019a}
. In this case, for
\begin{displaymath}
h = g \circ f =  x y a b x y + x y a b z + z a b x y + z a b z + c x y + c z + d e,
\end{displaymath}
the main ``structure'' becomes (almost) visible already after
\emph{minimization} of the corresponding ``polynomial''
admissible linear system (in upper triangular form with ones in the diagonal).

\newpage
\section{Application: Newton Iteration}\label{sec:fd.newton}

To compute the third root $\sqrt[3]{z}$ of a (positive) real number $z$,
say $z=2$,
we just have to find the roots of the polynomial $p = x^3 - z$,
for example by using the Newton iteration
$x_{k+1} := x_k - p(x_k)/p'(x_k)$, that 
---given a ``good'' starting value $x_0$--- yields
a (quadratically) convergent sequence $x_0, x_1, x_2, \ldots$
such that $\lim_{k \to \infty} x_k^3 = z$.
Using \textsc{FriCAS} \cite{FRICAS2019}
, the first 6~iterations for $x_0=1$ are
\begin{center}
\begin{tabular}{rd{8}d{16}}
$k$ & \multicolumn{1}{c}{$\lvert x_k - x_{k-1}\rvert$}
     & \multicolumn{1}{c}{$x_k$} \\\hline\tabstrut
1 & 3.333\cdot 10^{-1} & 1.3333333333333333 \\
2 & 6.944\cdot 10^{-2} & 1.2638888888888888 \\
3 & 3.955\cdot 10^{-3} & 1.259933493449977 \\
4 & 1.244\cdot 10^{-5} & 1.2599210500177698 \\
5 & 1.229\cdot 10^{-10} & 1.2599210498948732 \\
6 & 0 & 1.2599210498948732
\end{tabular}
\end{center}
Detailed discussions are available in every book on
numerical analysis, for example
\cite{Henrici1964a}
. As a starting point towards current research,
one could take \cite{Schleicher2017a}
.

\medskip
Now we would like to compute the third root $\sqrt[3]{Z}$
of a real (square) matrix $Z$ (with positive eigenvalues).
We (still) can use 
$X_{k+1} := \tsfrac{2}{3} X_k + \tsfrac{1}{3} X_k^{-2} Z$
if we choose a starting matrix $X_0$
such that $X_0 Z = Z X_0$ 
because then all $X_k$ \emph{commute} with $Z$
\cite[Section~7.3]{Higham2008a}
.
For $X_0=I$ and
\begin{displaymath}
Z =
\begin{bmatrix}
47 & 84 & 54 \\
42 & 116 & 99 \\
9 & 33 & 32
\end{bmatrix}
\quad
\text{respectively}
\quad
\sqrt[3]{Z} = 
\begin{bmatrix}
3 & 2 & 0 \\
1 & 4 & 3 \\
0 & 1 & 2
\end{bmatrix},
\end{displaymath}
some iterations are (with $\fronorm{.}$ denoting the Frobenius norm) 
\begin{center}
\begin{tabular}{rd{8}d{8}}
$k$ & \multicolumn{1}{c}{$\lVert X_k - X_{k-1}\rVert_{\text{F}}$}
    & \multicolumn{1}{c}{$\lVert X_k - \sqrt[3]{Z} \rVert_{\text{F}}$} \\\hline\tabstrut
4 & 9.874 & 23.679 \\
5 & 6.511 & 13.818 \\
6 & 4.109 & 7.309 \\
7 & 2.254 & 3.200 \\
8 & 8.295\cdot 10^{-1} & 9.455\cdot 10^{-1} \\
9 & 1.140\cdot 10^{-1} & 1.160\cdot 10^{-1} \\
10 & 2.044\cdot 10^{-3} & 2.044\cdot 10^{-3} \\
11 & 1.115\cdot 10^{-6} & 6.489\cdot 10^{-7}
\end{tabular}
\end{center}
Does this sequence of matrices $X_0,X_1,\ldots$ converge?
Some more iterations reveal immediately that something goes wrong,
visible in Table~\ref{tab:fd.newton}.
The problem is that due to rounding errors (in finite precision
arithmetics) commutativity of $X_k$ (with $Z$) is lost,
that is, $\fronorm{X_k Z - Z X_k}$ \emph{increases} with every
iteration. This happens even for a starting value close
to the solution $X_0 = \sqrt[3]{Z} + \eps I$.
For a detailled discussion of matrix roots
and how to overcome such problems
we refer to \cite[Section~7]{Higham2008a}
, for further information about numerics
to \cite{Demmel1997a}
. A classical introduction to matrix functions
is \cite[Chapter~V]{Gantmacher1965a}
. For the (matrix) square root and discussions about
the stability (of Newton's method) we recommend
\cite{Higham1986b}
.

\begin{table}[ht]
\begin{center}
\begin{tabular}{rd{8}d{8}d{8}}
$k$ & \multicolumn{1}{c}{$\fronorm{X_k - X_{k-1}}$}
    & \multicolumn{1}{c}{$\fronorm{X_k - \sqrt[3]{Z}}$}
    & \multicolumn{1}{c}{$\fronorm{X_k Z - Z X_k}$} \\\hline\tabstrut
1 & 65.836 & 5.385   & 2.274\cdot 10^{-13}\\
2 & 22.279 & 60.756  & 7.541\cdot 10^{-13}\\
3 & 14.845 & 38.500  & 1.110\cdot 10^{-12}\\
4 & 9.874 & 23.679   & 2.031\cdot 10^{-12}\\
5 & 6.511 & 13.818   & 7.725\cdot 10^{-12}\\
6 & 4.109 & 7.309    & 7.010\cdot 10^{-11}\\
7 & 2.254 & 3.200    & 9.451\cdot 10^{-10}\\
8 & 8.295\cdot 10^{-1} & 9.455\cdot 10^{-1} & 1.716\cdot 10^{-8}\\
9 & 1.140\cdot 10^{-1} & 1.160\cdot 10^{-1} & 3.548\cdot 10^{-7}\\
10 & 2.044\cdot 10^{-3} & 2.044\cdot 10^{-3} & 7.473\cdot 10^{-6}\\
11 & 1.115\cdot 10^{-6} & 6.489\cdot 10^{-7} & 1.574\cdot 10^{-4}\\
12 & 1.979\cdot 10^{-5} & 8.971\cdot 10^{-7} & 3.314\cdot 10^{-3}\\
13 & 4.168\cdot 10^{-4} & 1.890\cdot 10^{-5} & 6.978\cdot 10^{-2}\\
14 & 8.777\cdot 10^{-3} & 3.979\cdot 10^{-4} & 1.469 \\
15 & 1.848\cdot 10^{-1} & 8.379\cdot 10^{-3} & 3.094\cdot 10^{+1}\\
16 & 3.892              & 1.764\cdot 10^{-1} & 6.515\cdot 10^{+2}\\
17 & 8.195\cdot 10^{+1}   & 3.715 & 1.372\cdot 10^{+4}\\
18 & 1.726\cdot 10^{+3} & 7.823\cdot 10^{+1} & 2.889\cdot 10^{+5}
\end{tabular}
\end{center}
\vspace{-2ex}
\caption{The first 18 Newton iterations to find $\sqrt[3]{Z}$
for $X_0=I$. The values in column~2 (and~3) for the first 11 Newton
steps seem to indicate convergence. However, the (Frobenius) norm
of the \emph{commutator} $X_k Z - Z X_k$ (column~4) increases steadily and
that causes eventually a \emph{diverging} sequence $(X_k)$.
}
\label{tab:fd.newton}
\end{table}

\begin{remark}
To ensure commutation of the iterates $X_k$ with $Z$
one could use an additional correction step solving
the Sylvester equation $Z \Delta X_k - \Delta X_k Z= X_k Z - Z X_k$
for an update $\Delta X_k$.
Since this is expensive the benefit of using
the ``commutative'' Newton iteration would be lost.
For fast solutions of Sylvester equations we
refer to \cite{Kirrinnis2001a}
.
\end{remark}

\medskip
Although what we are going to introduce now
as ``\emph{non-commutative} (nc) \emph{Newton iteration}'' is even
more expensive, it can be implemented as \emph{black box} algorithm,
that is, \emph{without} manual computation of the non-commutative
derivative and any individual implementation/programming.
Furthermore, there is absolutely no restriction for the
initial iterate.

\medskip
To get the nc Newton method for solving $f(x)=0$ we just need to
``truncate'' the Hausdorff \emph{polarization operator}
\cite[Section~4]{Rota1980a}
\begin{displaymath}
f(x+b) = f(x) + \ncpd_{x|b} f(x)
  + \tsfrac{1}{2}\ncpd_{x|b}\bigl(\ncpd_{x|b} f(x) \bigr) + \ldots
\end{displaymath}
as analogon to Taylor's formula. Thus, for $f(x) = x^3 - z$ we
have to solve
\begin{displaymath}
0 = \underbrace{x^3 - z}_{f} + \underbrace{bx^2 + xbx + x^2b}_{\ncpd_{x|b} f}
    \quad \in \freeFLD{\numR}{\alphabet{X}} =: \field{F}
\end{displaymath}
with respect to $b$. In terms of (square) matrices $X,Z,B$ this
is just the generalized Sylvester equation
$B X^2 + X B X + X^2 B = Z - X^3$
which is \emph{linear} in $B$
\cite[Section~7.2]{Higham2008a}
. Taking a (not necessarily with $Z$ commuting) starting matrix $X_0$,
we can compute $B_0$ and get $X_1 := X_0 + B_0$ and iteratively
$X_{k+1} := X_k + B_k$.

Minimal admissible linear systems for $f=x^3-z$ and
$\ncpd_{x|b} f=bx^2 + xbx + x^2b$ are given by
\vspace{-3ex}
\begin{displaymath}
\begin{bmatrix}
1 & -x & . & z \\
. & 1 & -x & . \\
. & . & 1 & -x \\
. & . & . & 1
\end{bmatrix}
s =
\begin{bmatrix}
. \\ . \\ . \\ 1
\end{bmatrix}
\quad\text{and}\quad
\begin{bmatrix}
1 & -x & . & -b & . & . \\
. & 1 & -x & . & -b & . \\
. & . & 1 & . & . & -b \\
. & . & . & 1 & -x & . \\
. & . & . & . & 1 & -x \\
. & . & . & . & . & 1
\end{bmatrix}
s =
\begin{bmatrix}
. \\ . \\ . \\ . \\ . \\ 1
\end{bmatrix}
\end{displaymath}
respectively. A \emph{minimal} ALS 
$\als{A}=(u,A,v)$ of dimension $n=6$
for $g = f+ \ncpd_{x|b}f$ is immediate:
\begin{displaymath}
\begin{bmatrix}
1 & -x & . & -b & . & z \\
. & 1 & -x & . & -b & . \\
. & . & 1 & . & . & -b-x \\
. & . & . & 1 & -x & . \\
. & . & . & . & 1 & -x \\
. & . & . & . & . & 1
\end{bmatrix}
s =
\begin{bmatrix}
. \\ . \\ . \\ . \\ . \\ 1
\end{bmatrix}.
\end{displaymath}
Since the system matrix is just a linear matrix pencil
$A = A_1 \otimes 1 + A_b \otimes b + A_x \otimes x + A_z \otimes z$
we can plug in $m \times m$ matrices using the Kronecker product
\begin{displaymath}
A(B,X,Z) 
  = A_1 \otimes I + A_b \otimes B + A_x \otimes X + A_z \otimes Z.
\end{displaymath}
In this case, for $v = e_n = [0, \ldots, 0, 1]\trp$,
the evaluation of $g: M_m(\numR)^3 \to M_m(\numR)$
is just the upper right $m \times m$ block of $A(B,X,Z)\inv$.
For how to efficiently evaluate (nc) polynomials by matrices
using \emph{Horner systems} we refer to
\cite{Schrempf2019a}
.

Here, $A(B,X,Z)$ is invertible for arbitrary $B$, $X$ and $Z$.
Given $Z$ and $X$,
to find a $B$ such that $g(B,X,Z)=0$
we have to look for transformation
matrices $P=P(T_{ij})$ and $Q=Q(U_{ij})$ of the form
\begin{displaymath}
P =
\begin{bmatrix}
I & . & . & T_{1,1} & T_{1,2} & 0 \\
. & I & . & T_{2,1} & T_{2,2} & 0 \\
. & . & I & T_{3,1} & T_{3,2} & 0 \\
. & . & . & I & . & . \\
. & . & . & . & I & . \\
. & . & . & . & . & I
\end{bmatrix},
\quad
Q =
\begin{bmatrix}
I & . & . & 0 & 0 & 0 \\
. & I & . & U_{2,1} & U_{2,2} & U_{2,3} \\
. & . & I & U_{3,1} & U_{3,2} & U_{3,3} \\
. & . & . & I & . & . \\
. & . & . & . & I & . \\
. & . & . & . & . & I
\end{bmatrix}
\end{displaymath}
such that the upper right block of size $3m \times 3m$ of
$P A(B,X,Z) Q$ is zero, that is,
we need to solve a \emph{linear} system of equations
(with $6m^2+6m^2+m^2 = 13m^2$ unknowns)
similar to the word problem
\cite[Theorem~2.4]{Schrempf2017a9}
. Table~\ref{tab:fd.ncnewton} shows the nc
Newton iterations for the starting matrix
\begin{equation}\label{eqn:fd.ncnewton.x0}
X_0 =
\begin{bmatrix}
1 & 0 & 2 \\
0 & 1 & 0 \\
0 & 0 & 1
\end{bmatrix}
\end{equation}
using \textsc{FriCAS} \cite{FRICAS2019}
\ and the least squares solver DGELS from
\cite{LAPACK2018}
.

\begin{table}
\begin{center}
\begin{tabular}{rd{8}d{8}d{8}}
$k$ & \multicolumn{1}{c}{$\fronorm{B_k}$}
    & \multicolumn{1}{c}{$\fronorm{X_k - \sqrt[3]{Z}}$}
    & \multicolumn{1}{c}{$\fronorm{X_k Z - Z X_k}$} \\\hline\tabstrut
 0 & 46.877 &  5.745 &  113.842 \\
 1 & 16.081 & 42.298 & 5552.242 \\
 2 & 10.768 & 26.374 & 3659.788 \\
 3 &  7.320 & 15.912 & 2395.971 \\
 4 &  4.971 &  9.358 & 1534.892 \\
 5 &  2.934 &  6.122 & 912.700 \\
 6 &  2.651 &  4.389 & 414.201 \\
 7 &  1.380 &  1.846 & 86.771 \\
 8 & 3.878\cdot 10^{-1}  & 4.875\cdot 10^{-1} & 5.638 \\
 9 & 9.378\cdot 10^{-2}  & 1.023\cdot 10^{-1} & 2.652\cdot 10^{-1} \\
10 & 8.432\cdot 10^{-3}  & 8.506\cdot 10^{-3} & 6.737\cdot 10^{-3} \\
11 & 7.407\cdot 10^{-5}  & 7.408\cdot 10^{-5} & 1.951\cdot 10^{-5} \\
12 & 5.895\cdot 10^{-9}  & 5.895\cdot 10^{-9} & 5.106\cdot 10^{-10} \\
13 & 1.825\cdot 10^{-14} & 1.521\cdot 10^{-14} & 1.491\cdot 10^{-12}
\end{tabular}
\end{center}
\vspace{-2ex}
\caption{The first 13 (nc) Newton iterations to find $\sqrt[3]{Z}$
  for $X_0$ from \eqref{eqn:fd.ncnewton.x0}.
  In the beginning, the iterates $X_k$ do not commute with $Z$
  (column~4).
  }
\label{tab:fd.ncnewton}
\end{table}

\medskip
In general, depending on the initial iterate $X_0$,
there is \emph{no} guarantee that one ends up in some
prescribed solution $X$ since there can be several
(matrix) roots.
For the \emph{principal} $p$-th root (and further discussion)
we refer to \cite[Theorem~7.2]{Higham2008a}
.
For a discussion of Taylor's theorem (for matrix functions)
one should have a look in 
\cite{Deadman2016a}
.

\medskip
The goal of this section was mainly for illustration
(of the use of the free derivative).
How could we attack analysis of convergence (for special classes
of rational functions) in this context?
``Classical'' interval Newton is discussed in
\cite[Chapter~6.1]{Rump2010a}
. What's about ``nc interval Newton''?
The next natural step would be \emph{multi-dimensional} nc Newton,
say to find a root $G(X,Y,Z)=0$ for
\begin{displaymath}
G =
\begin{bmatrix}
g_1 \\ g_2 \\ g_3
\end{bmatrix}
=
\begin{bmatrix}
x^3 + z -d \\
\bigl(x(1-yx)\bigr)\inv -e \\
xzx - yz + z^2 - f
\end{bmatrix}
\end{displaymath}
with (matrix-valued) parameters $d,e,f$.
Although very technical, one can set up a ``joint'' linear
system of equations to solve
\begin{displaymath}
\begin{pmatrix}
g_1 + \ncpd_{x\mid a} g_1 & g_1 + \ncpd_{y\mid b} g_1 & g_1 + \ncpd_{z\mid c} g_1 \\
g_2 + \ncpd_{x\mid a} g_2 & g_2 + \ncpd_{y\mid b} g_2 & g_2 + \ncpd_{z\mid c} g_2 \\
g_3 + \ncpd_{x\mid a} g_3 & g_3 + \ncpd_{y\mid b} g_3 & g_3 + \ncpd_{z\mid c} g_3
\end{pmatrix} = 0
\end{displaymath}
for matrices $A$, $B$ and $C$
using the previous approach to find tuples of ``individual''
transformation matrices $P_{ij}$ and $Q_{ij}$ to create
respective upper right blocks of zeros.
The problem however is, that this system is \emph{overdetermined}
in general, and starting arbitrary close to a root
leads to a residual that causes divergence.
How can one overcome that?

\section*{Acknowledgement}

I am very grateful to Tobias Mai for making me aware of cyclic derivatives
respectively the work of Rota, Sagan and Stein
\cite{Rota1980a}
\ in May 2017 in Graz and for a
discussion in October 2018 in Saarbrücken.
I also use this opportunity to thank Soumyashant Nayak
for continuous ``non-commutative'' discussions,
in particular about free associative algebras
and (inner) derivations.
And I thank the referees for the feedback and a hint on the
literature.

\bibliographystyle{alpha}
\bibliography{doku}
\addcontentsline{toc}{section}{Bibliography}

\end{document}